\documentclass[11pt]{article}

\usepackage[T1]{fontenc}
\usepackage[margin=2.5cm]{geometry}
\usepackage{graphicx}

\usepackage{amsthm}
\usepackage{amsmath}
\usepackage{amssymb}

\usepackage{mathtools}
\usepackage{enumitem}

\usepackage{hyperref}
\usepackage{cleveref}

\usepackage{algorithm}
\usepackage{algpseudocode}

\algrenewcommand\algorithmicrequire{\textbf{Input:}}
\algrenewcommand\algorithmicensure{\textbf{Output:}}
\algnewcommand\algorithmicforeach{\textbf{for each}}
\algdef{S}[FOR]{ForEach}[1]{\algorithmicforeach\ #1\ \algorithmicdo}

\newtheorem{theorem}{Theorem}[section]
\newtheorem{lemma}[theorem]{Lemma}
\newtheorem{proposition}[theorem]{Proposition}
\newtheorem{corollary}[theorem]{Corollary}
\newtheorem{definition}[theorem]{Definition}
\newtheorem{remark}[theorem]{Remark}

\newtheorem{question}{Question}

\usepackage{tikz}
\usetikzlibrary{arrows.meta}

\tikzset{
  vtx/.style = {circle, draw, fill=white, minimum size=4.8mm,
                inner sep=0pt, font=\tiny},
  sel/.style = {vtx, fill=black!70, text=white, font=\tiny\bfseries},
  ann/.style = {font=\tiny, inner sep=1.2pt},
  frc/.style = {-{Latex[length=1.6mm,width=1.2mm]}, semithick,
                shorten >=1.0pt, shorten <=1.0pt}
}

\usepackage[english]{babel}
\usepackage{csquotes}

\usepackage[
    sorting=ynt,
    maxbibnames=99,
    minbibnames=99,
    maxcitenames=99,
    mincitenames=99
]{biblatex}

\usepackage{authblk}

\usepackage{xcolor}

\title{Fort Abundance in Zero Forcing}

\author[1,2]{Aida Abiad}
\author[3,4]{Sina Ghasemi Nezhad}

\affil[1]{Department of Mathematics and Computer Science, Eindhoven University of Technology, The Netherlands}
\affil[2]{Department of Mathematics and Data Science, Vrije Universiteit Brussel, Belgium}

\affil[3]{Department of Computer Science, Vrije Universiteit Amsterdam, The Netherlands}
\affil[4]{Informatics Institute, Faculty of Science, University of Amsterdam, The Netherlands}

\affil[ ]{\texttt{a.abiad.monge@tue.nl}, \texttt{s.ghaseminezhad@student.vu.nl}}

\date{}

\begin{document}

\maketitle


\begin{abstract}
This paper paper concerns the study of forts, the sets that obstruct zero forcing. We show that every block graph on $n$ vertices has at least $n/3$ minimal forts, extending a recent bound for trees by Cameron and and Li (arXiv:2512.12874). The number of forts, and with it the number of compatible collections, grows exponentially for every tree, every connected non-star graph of bounded degree, and every connected graph of linear minimum degree, answering a  question in the negative by Hicks et al. (INFORMS Journal on Computing 2022) for these graph classes. We finish by counting the minimal forts of a tree exactly, in linear time and space.\\


\end{abstract}

\section{Introduction}\label{sec:intro}

Zero forcing was introduced to bound the maximum nullity of symmetric matrices described by a graph~\cite{aim2008zero,hicks2022computational,brimkov2021improved} and was found independently in quantum information theory as graph infection~\cite{burgarth2007full} and in graph searching as fast-mixed search~\cite{yang2013fast,fallat2016complexity}. Let $G=(V,E)$ be a graph and let $S\subseteq V$ be a set of initially blue vertices, the remaining vertices being white. The \emph{color change rule} allows a blue vertex $u$ to \emph{force} a white vertex $v$ to become blue when $v$ is the unique white neighbor of $u$. Blue vertices remain blue. The set $S$ is a \emph{zero forcing set} if repeated application of the rule colors every vertex of $G$ blue, and the \emph{zero forcing number} $Z(G)$ is the minimum cardinality of a zero forcing set. Computing $Z(G)$ is known to be NP-hard~\cite{aazami2008hardness}. 

In this paper we investigate the obstruction structure of zero forcing. A \emph{fort} is a non-empty set of vertices no vertex outside which has exactly one neighbor inside; a fort is \emph{minimal} if none of its proper subsets is a fort. Forts characterize zero forcing sets exactly: a set $S$ is a zero forcing set if and only if its complement contains no fort~\cite{brimkov2019computational}. Minimal forts therefore drive the set covering and integer programming formulations used in computational zero forcing, and they carry a growing structural load, for instance in bounding zero forcing numbers of graph products and in the study of maximum nullity through compatible collections of forts~\cite{hicks2022computational,furst2025compatible}.

We contribute with three results in this direction. The first extends the recent lower bound of Cameron and Li~\cite{cameron2025minimal} from trees to all block graphs: every block graph on $n$ vertices has at least $n/3$ minimal forts (Theorem~\ref{thm:block-forts}). The second takes up Question~1 of Hicks et al.~\cite{hicks2022computational}, which asks which graphs admit only polynomially many compatible sets of minimal forts. Furst, Hutchens, Mitchell, and Zhang~\cite{furst2025compatible} showed that the minimum-rank bound motivating that question in fact requires compatible collections of \emph{arbitrary} forts, and it is the corrected question that we study. The number of forts, and hence the number of compatible collections of forts, grows exponentially for every tree, every connected non-star graph of bounded maximum degree, and every connected graph of minimum degree linear in $n$, so the corrected question has a negative answer on each of these classes (Corollary~\ref{cor:resolution}). Whether families of unbounded maximum degree and sublinear minimum degree can behave differently we do not know. This should be contrasted with the count of \emph{minimal} forts, which can be polynomial and whose extremal behaviour was studied by Becker, Cameron, Hanely, Ong, and Previte~\cite{becker2025number}. The third is an exact counting algorithm; dynamic programming over a rooted tree computes the number of minimal forts of any tree in linear time and linear space (Theorem~\ref{thm:dp-alg}), complementing the bounds and the enumeration approach of Nguyen and Kenter~\cite{dat2026number}.

\noindent\textbf{Notation.} We write $\operatorname{Fort}(G)$ to denote the family of forts of $G$, and $\operatorname{Fort}_{\min}(G)$ for its inclusion-minimal members. A nonempty collection $\mathcal{C}\subseteq\operatorname{Fort}(G)$ is \emph{compatible} if, for every two distinct $F_1,F_2\in\mathcal{C}$ and every $v\in F_1\cap F_2$, there exists $F_3\in\mathcal{C}$ such that $ F_3\subseteq (F_1\cup F_2)\setminus\{v\}$. Let
$$
    f(G)=|\operatorname{Fort}(G)|, \qquad f_m(G)=|\operatorname{Fort}_{\min}(G)|,
$$
and let $\operatorname{cs}(G)$ and $\operatorname{cs}_m(G)$ denote the numbers of nonempty compatible sub-collections of $\operatorname{Fort}(G)$ and $\operatorname{Fort}_{\min}(G)$, respectively.

\section{Zero Forcing Minimal Forts of Block Graphs}

Cameron and Li~\cite{cameron2025minimal} proved that every tree of order $n$ has at least $n/3$ minimal forts.

\begin{lemma}\cite{cameron2025minimal}\label{lem:tree-forts}
    If $T$ is a tree of order $n$, then $T$ has at least $n/3$ minimal forts.
\end{lemma}

We show that in fact the same lower bound holds when considering the larger class of graphs of block graphs.

\begin{lemma}\label{lem:fort-components}
    Let $G_1,\ldots,G_k$ be the connected components of $G$. Then
    $$
        f_m(G)=\sum_{i=1}^{k}f_m(G_i).
    $$
\end{lemma}

\begin{proof}
    Let $F$ be a fort of $G$. For every $i$ such that $F\cap V(G_i)\neq\varnothing$, the set $F\cap V(G_i)$ is a fort of $G_i$, because there are no edges between different components. Consequently, if $F$ is a minimal fort of $G$, then it is contained in exactly one component and is a minimal fort of that component. Conversely, every minimal fort of a component is a minimal fort of $G$.
\end{proof}

\begin{lemma}\label{lem:rooted-alternative}
    Let $H$ be a connected graph and let $d\in V(H)$. At least one of the following statements holds.
    \begin{enumerate}
        \item The vertex $d$ belongs to a minimal fort of $H$.
        \item There is a nonempty set $Q\subseteq V(H)\setminus\{d\}$ such that
        \begin{enumerate}
            \item $|N_H(d)\cap Q|=1$;
            \item every vertex $v\in V(H)\setminus(Q\cup\{d\})$ satisfies $|N_H(v)\cap Q|\neq 1$;
            \item $Q$ contains no fort of $H$.
        \end{enumerate}
    \end{enumerate}
\end{lemma}

\begin{proof}
    Suppose that $d$ belongs to no minimal fort of $H$. Thus $d$ is fort irrelevant. By a theorem of Jacob~\cite[Thm.~2.4]{jacob2025well}, fort irrelevance and zero-forcing irrelevance are equivalent. Hence $d$ belongs to no minimal zero forcing set and, in particular, to no minimum zero forcing set.

    Let $Z$ be a minimum zero forcing set of $H$. Starting from $Z$, perform forces so that a vertex different from $d$ is used whenever such a force is available; allow $d$ to force only when no other vertex can force. Since $Z$ is a zero forcing set, this process eventually colors every vertex.

    The vertex $d$ must perform a force. Otherwise, $d$ is a terminal vertex of its forcing chain. Reversing all forcing chains would then produce a minimum zero forcing set containing $d$, contrary to the choice of $d$.

    Let $Q$ be the set of uncolored vertices immediately before the first force performed by $d$. At this stage $d$ has exactly one neighbor in $Q$, while every other colored vertex has either zero or at least two neighbors in $Q$. The colored set $V(H)\setminus Q$ is a zero forcing set, since the remaining forcing sequence colors all vertices. Every zero forcing set intersects every fort, so $Q$ contains no fort of $H$.
\end{proof}

\begin{lemma}\label{lem:pendant-star}
    Let $u\in V(G)$ have $p\geq 2$ pendant neighbors $x_1,\ldots,x_p$. Let $H=G-\{u,x_1,\ldots,x_p\}$, and let $H_1,\ldots,H_r$ be the nonempty connected components of $H$. Then
    $$
        f_m(G)\geq \sum_{j=1}^{r}f_m(H_j)+\binom{p}{2}.
    $$
\end{lemma}

\begin{proof}
    Let $F$ be a minimal fort of some component $H_j$, and put $k=|N_G(u)\cap F|$. If $k\neq 1$, then $F$ is a minimal fort of $G$. If $k=1$, then $F\cup\{x_1\}$ is a minimal fort of $G$. Indeed, $u$ has two neighbors in $F\cup\{x_1\}$, and every other vertex outside this set satisfies the fort condition.

    To verify minimality in the second case, let $K\subseteq F\cup\{x_1\}$ be a fort of $G$. If $x_1\in K$, then $K\cap V(H_j)\neq\varnothing$, since otherwise $u$ has exactly one neighbor in $K$. Moreover, $K\cap V(H_j)$ is a fort of $H_j$, so the minimality of $F$ implies $K\cap V(H_j)=F$. If $x_1\notin K$, then $K\subseteq F$ and $K$ is a fort of $H_j$; hence $K=F$, but $F$ is not a fort of $G$ when $k=1$. Thus $F\cup\{x_1\}$ is minimal.

    These constructions are distinct for different minimal forts and different components. In addition, every pair $\{x_i,x_j\}$, $i\neq j$, is a minimal fort of $G$. Therefore
    $$
        f_m(G)\geq \sum_{j=1}^{r}f_m(H_j)+\binom{p}{2},
    $$
    which finishes the proof.
\end{proof}

\begin{lemma}\label{lem:terminal-p3}
    Suppose that $G$ contains an induced path $d-a-b-c$ such that $N_G(a)=\{d,b\}$, $N_G(b)=\{a,c\}$, and $N_G(c)=\{b\}$. Let $H=G-\{a,b,c\}$, and assume that $H$ is connected. Then
    $$
        f_m(G)\geq f_m(H)+1.
    $$
\end{lemma}

\begin{proof}
    Partition the minimal forts of $H$ into
    $$
        \mathcal{A} = \{F\in\operatorname{Fort}_{\min}(H):d\notin F\},
        \qquad
        \mathcal{B} = \{F\in\operatorname{Fort}_{\min}(H):d\in F\}.
    $$
    Every $F\in\mathcal{A}$ remains a minimal fort of $G$. For every $F\in\mathcal{B}$, the set $\widehat F=F\cup\{b,c\}$ is a minimal fort of $G$. It is a fort because $a$ has the two neighbors $d$ and $b$ in $\widehat F$.

    To prove minimality, let $K\subseteq\widehat F$ be a fort of $G$. The set $K$ contains either both $b,c$ or neither: if exactly one of them is selected, then either $b$ or $c$ has exactly one selected neighbor. If $b,c\in K$, then $d\in K$, since otherwise $a$ has the unique selected neighbor $b$. It follows that $K\cap V(H)$ is a fort of $H$, and hence $K\cap V(H)=F$. If $b,c\notin K$, then $d\notin K$, since otherwise $a$ has exactly one selected neighbor. In that case $K\subseteq F\setminus\{d\}$ would be a fort of $H$, which is impossible. Thus $K=\widehat F$.

    The constructions above give $f_m(H)$ distinct minimal forts of $G$. We construct one more. If $\mathcal{B}\neq\varnothing$, choose $F_0\in\mathcal{B}$ and set $X=F_0\cup\{a,c\}$.

    Then $X$ is a fort of $G$, because $b$ has the two neighbors $a,c$ in $X$. The set $X$ contains no fort lying entirely in $H$: no proper subset of $F_0$ is a fort of $H$, while $F_0$ itself is not a fort of $G$, since $a$ has the unique neighbor $d$ in $F_0$. It also contains no fort lying entirely in $\{a,c\}$. Therefore, every minimal fort contained in $X$ intersects both $V(H)$ and $\{a,c\}$. Since $b\notin X$, such a minimal fort must contain both $a$ and $c$, and hence it is different from all previously constructed minimal forts.

    If $\mathcal{B}=\varnothing$, let $Q$ be supplied by Lemma~\ref{lem:rooted-alternative} and set $X=Q\cup\{a,c\}$. Now $d$ has its unique neighbor in $Q$ together with the neighbor $a$, and $b$ has the two neighbors $a,c$. Thus $X$ is a fort. Since $Q$ contains no fort of $H$, the same argument gives a new minimal fort contained in $X$. Hence $f_m(G)\geq f_m(H)+1$.
\end{proof}

\begin{lemma}\label{lem:terminal-hub}
    Let $H$ be a connected graph with a specified vertex $d$. Form $G$ by adding a vertex $u$, the edge $du$, and $t\geq2$ pairwise internally disjoint pendant paths
    $$
        P_i=u-y_{i,1}-\cdots-y_{i,\ell_i}, \qquad \ell_i\in\{1,2\},
    $$
    such that the added vertices have no other neighbors. Assume that at
    most one of the paths has length one. Let $W=1+\sum_{i=1}^{t}\ell_i$
    be the number of added vertices. Then
    $$
        f_m(G)\geq f_m(H)+\frac{W}{3}.
    $$
\end{lemma}

\begin{proof}
    For each $i$, write $A_i=V(P_i)\setminus\{u\}$. Every pair $A_i\cup A_j$, $i\neq j$, is a minimal fort of $G$. Indeed, $u$ has exactly two neighbors in this set, and on a length-two arm either both arm vertices must be selected or neither. Thus these sets give $\binom{t}{2}$ distinct minimal forts.

    Every minimal fort $F$ of $H$ avoiding $d$ remains a minimal fort of $G$. If $d\in F$, then $F\cup A_1$ is a minimal fort of $G$. The vertex $u$ has the two neighbors $d$ and $y_{1,1}$. Moreover, any fort contained in $F\cup A_1$ that meets $A_1$ must contain all of $A_1$, must contain $d$, and has a fort of $H$ as its intersection with $V(H)$. The minimality of $F$ therefore proves the claim.

    Consequently,
    $$
        f_m(G)\geq f_m(H)+\binom{t}{2}.
    $$
    If $t\geq3$, then $\binom{t}{2}\geq\frac{1+2t}{3}\geq\frac{W}{3}$, and the result follows.

    It remains to consider $t=2$. Since at most one arm has length one, at least one arm has length two; relabel so that $P_1$ has length two. Let $e_i=y_{i,\ell_i}$ be the terminal vertex of $P_i$, and put $Y=\{u,e_1,e_2\}$,

    If $d$ belongs to a minimal fort $F_0$ of $H$, then $F_0\cup Y$ is a fort of $G$. If $d$ belongs to no minimal fort, let $Q$ be supplied by Lemma~\ref{lem:rooted-alternative}; then $Q\cup Y$ is a fort of $G$. In the latter case, $d$ has its unique neighbor in $Q$ together with $u$. In either case, the middle vertex of every length-two arm has the two neighbors $u$ and $e_i$.

    The set $Y$ contains no fort supported entirely on the added vertices: if $u$ is selected, then $d$ has exactly one selected neighbor, while if $u$ is not selected, a selected terminal vertex on a length-two arm leaves its middle vertex with exactly one selected neighbor. The possible length-one arm cannot by itself form a fort. The chosen subset of $H$ also contains no fort that is a fort of $G$. Hence the displayed fort contains a minimal fort that meets both $H$ and the added vertices. This new minimal fort is different from the arm-pair fort, from every minimal fort lying in $H$, and from every lifted fort $F\cup A_1$, because the latter contains the middle vertex of the length-two arm $P_1$, whereas $Y$ does not.

    Thus, when $t=2$,
    $$
        f_m(G) \geq f_m(H)+2 \geq f_m(H)+\frac{W}{3},
    $$
    because $W\leq 5$.
\end{proof}

\begin{lemma}\label{lem:terminal-clique}
    Let $H$ be a connected graph with a specified vertex $d$. Form $G$ by adding vertices $z_1,\ldots,z_s$, where $s\geq2$, and making $B=\{d,z_1,\ldots,z_s\}$ a clique. At each $z_i$, attach $t_i\geq0$ pairwise internally disjoint pendant paths
    $$
        P_{i,j} = z_i-y_{i,j,1}-\cdots-y_{i,j,\ell_{i,j}}, \qquad \ell_{i,j}\in\{1,2\},
    $$
    and assume that the added vertices have no further neighbors. Assume also that, for each $i$, at most one arm at $z_i$ has length one. Let $W$ be the number of added vertices. Then
    $$
        f_m(G) \geq f_m(H)+\frac{W}{3}.
    $$
\end{lemma}

\begin{proof}
    For every arm $P_{i,j}$, let $A_{i,j}=V(P_{i,j})\setminus\{z_i\}$, and define the core package $C_i = \{z_i\} \cup \{y_{i,j,\ell_{i,j}}:1\leq j\leq t_i\}$. Thus $C_i$ contains $z_i$ and the terminal vertex of each arm at $z_i$, but not the middle vertex of a length-two arm.

    We first construct $f_m(H)$ minimal forts of $G$. Every minimal fort of $H$ avoiding $d$ remains a minimal fort of $G$. If $F\in\operatorname{Fort}_{\min}(H)$ contains $d$, then $F\cup C_1$ is a minimal fort of $G$. Indeed, every vertex $z_j\in B\setminus\{d,z_1\}$ has the two selected neighbors $d,z_1$, and the middle vertex of a selected length-two arm at $z_1$ has the two selected neighbors $z_1$ and its terminal vertex.

    For minimality, let $K\subseteq F\cup C_1$ be a fort. If $z_1\in K$, then $d\in K$, since any $z_j\in B\setminus\{d,z_1\}$ would otherwise have the unique selected neighbor $z_1$. All terminal vertices in $C_1$ are then forced into $K$: omitting a pendant terminal vertex leaves that vertex with the unique selected neighbor $z_1$, while omitting the terminal vertex of a length-two arm leaves its middle vertex with the unique selected neighbor $z_1$. It follows that $K\cap V(H)$ is a fort of $H$, so $K=F\cup C_1$. If $z_1\notin K$, then $d\notin K$, and no terminal vertex in $C_1$ can belong to $K$. Thus $K\subseteq F\setminus\{d\}$ would be a fort of $H$, a contradiction.

    Next, for every $i\neq j$, the set $C_i\cup C_j$ is a minimal fort of $G$. Any vertex of the clique outside this set has the two selected neighbors $z_i,z_j$, and the middle vertex of a selected length-two arm has its core vertex and terminal vertex as selected neighbors. Since $|B|\geq3$, every fort contained in $C_i\cup C_j$ must contain either both $z_i,z_j$ or neither. The latter is impossible for a nonempty fort, and once $z_i,z_j$ are selected, every terminal vertex in the two packages is forced. Hence there are $\binom{s}{2}$ such minimal forts.

    Finally, for every pair of distinct arms $P_{i,j},P_{i,k}$ at the same core vertex $z_i$, the set $A_{i,j}\cup A_{i,k}$ is a minimal fort. This gives $\sum_{i=1}^{s}\binom{t_i}{2}$ additional minimal forts. All the constructions above are distinct. Therefore,
    $$
        f_m(G) \geq f_m(H) + L,
        \qquad
        L = \binom{s}{2} + \sum_{i=1}^{s}\binom{t_i}{2}.
    $$

    Put $T=\sum_{i=1}^{s}t_i$. Since every arm has length at most two, $W\leq s+2T$. For every nonnegative integer $t$, we have $3\binom{t}{2}\geq 2t-2$. If $s\geq3$, then
    $$
        3L \geq 3\binom{s}{2}+2T-2s \geq s+2T \geq W,
    $$
    where the second inequality follows from $3\binom{s}{2}\geq3s$.

    It remains to consider $s=2$. If $T=0$, then $W=2$ and $L=1$, so $L\geq W/3$. Suppose that $T\geq1$, and relabel so that $t_1\geq1$. Choose one arm $P_{1,1}$ at $z_1$ and put $Y=C_2\cup A_{1,1}$.

    If $d$ belongs to a minimal fort $F_0$ of $H$, then $F_0\cup Y$ is a fort of $G$. If $d$ belongs to no minimal fort of $H$, let $Q$ be supplied by Lemma~\ref{lem:rooted-alternative}; then $Q\cup Y$ is a fort of $G$. In the second case, $d$ has its unique selected neighbor in $Q$ together with $z_2$, while $z_1$ has the two selected neighbors $z_2$ and $y_{1,1,1}$.

    The set $Y$ contains no fort supported entirely on the added vertices. Indeed, because $d$ is adjacent to $z_2$ but not $z_1$, such a fort cannot contain $z_2$. Once $z_2$ is omitted, no terminal vertex from $C_2$ can be selected, and a single arm $A_{1,1}$ does not contain a fort. The selected subset of $H$ also contains no fort that is a fort of $G$. Hence the displayed fort contains a minimal fort meeting both $H$ and the added vertices. This minimal fort is new: it is not supported entirely in $H$, it is not one of the local forts, and it is not a lifted fort $F\cup C_1$, since $z_1\notin Y$.

    We have therefore obtained
    $$
        f_m(G)\geq f_m(H)+L+1.
    $$
    Moreover,
    $$
        3(L+1) = 6+3\binom{t_1}{2}+3\binom{t_2}{2} \geq 6+(2t_1-2)+(2t_2-2) = 2+2T \geq W.
    $$
    This completes the proof.
\end{proof}

Now we are ready to state one of the main results of this section.

\begin{theorem}\label{thm:block-forts}
    Let $G$ be a block graph on $n$ vertices. Then $G$ has at least $n/3$ minimal forts.
\end{theorem}

\begin{proof}
    We proceed by strong induction on $n$. The result is immediate for $n=1$.

    Suppose first that $G$ is disconnected, with nonempty connected components $G_1,\ldots,G_k$. By Lemma~\ref{lem:fort-components} and the induction hypothesis,
    $$
        f_m(G) = \sum_{i=1}^{k}f_m(G_i) \geq \sum_{i=1}^{k}\frac{|V(G_i)|}{3} = \frac{n}{3}.
    $$
    We may therefore assume that $G$ is connected.

    If every block of $G$ has order two, then $G$ is a tree, and the result follows from Lemma~\ref{lem:tree-forts}. Hence assume that $G$ has a block of order at least three.

    If some vertex $u$ has $p\geq2$ pendant neighbors, apply Lemma~\ref{lem:pendant-star}. The graph obtained after deleting $u$ and its pendant neighbors has total order $n-p-1$. Thus, by the induction hypothesis,
    $$
        f_m(G) \geq \frac{n-p-1}{3}+\binom{p}{2} \geq \frac{n}{3},
    $$
    because $\binom{p}{2} \geq \frac{p+1}{3}$ for $p \geq 2$. We may therefore assume that no vertex of $G$ has two pendant neighbors.

    Choose a block $R$ of order at least three and root the block--cut tree of $G$ at $R$. Let $B$ be a block of order at least three at maximum distance from $R$. By the choice of $B$, every block strictly below $B$ is a bridge. Consequently, the subgraphs descending from vertices of $B$ are ordinary rooted trees.

    Consider a terminal non-branching path in one of these descendant trees. If it has at least three vertices after its attachment vertex, then its last three vertices form a path $d-a-b-c$ satisfying the hypotheses of Lemma~\ref{lem:terminal-p3}. Deleting $a,b,c$ leaves a connected block graph of order $n-3$. Hence
    $$
        f_m(G) \geq \frac{n-3}{3}+1 = \frac{n}{3}.
    $$

    We may therefore assume that every terminal non-branching path below $B$ has length at most two. Suppose that a vertex $u$ strictly below $B$ has at least two children in the rooted descendant tree. Choose $u$ at maximum distance from $B$. Its descendant branches are pairwise disjoint pendant paths of lengths one or two. At most one has length one, because otherwise $u$ would have two pendant neighbors. Thus the configuration at $u$ satisfies Lemma~\ref{lem:terminal-hub}. Deleting $u$ and all its descendants leaves a connected block graph $H$, and the induction hypothesis together with Lemma~\ref{lem:terminal-hub} gives
    $$
        f_m(G) \geq f_m(H)+\frac{|V(G)\setminus V(H)|}{3} \geq \frac{n}{3}.
    $$

    It remains to consider the case in which no vertex strictly below $B$ branches. Hence, at each vertex of $B$, the descendant subgraph is a collection of pendant paths of lengths one or two attached directly to that vertex. Again, at most one such path at each vertex has length one.

    Suppose first that $B \neq R$. Let $d$ be the unique cut vertex of $B$ lying on the path from $B$ to $R$, and write $B=\{d,z_1,\ldots,z_s\}$. Since $|B| \geq 3$, we have $s \geq 2$. Delete $B\setminus\{d\}$ and all descendant paths attached to its vertices. The remaining graph $H$ is a connected block graph, and the deleted part is precisely the terminal-clique configuration of Lemma~\ref{lem:terminal-clique}. Therefore,
    $$
        f_m(G) \geq f_m(H)+\frac{|V(G)\setminus V(H)|}{3} \geq \frac{|V(H)|}{3}+\frac{|V(G)\setminus V(H)|}{3} = \frac{n}{3}.
    $$

    Finally, suppose that $B=R$. Then $R$ is the only block of order at least three. Thus $G$ consists of the clique $R=\{z_1,\ldots,z_s\}$ for $s\geq3$, together with collections of pendant paths of lengths one or two attached to its vertices. Use the notation $t_i$, $T$, $C_i$, and $A_{i,j}$ from Lemma~\ref{lem:terminal-clique}. The sets $C_i\cup C_j$ with $i\neq j$, and $A_{i,j}\cup A_{i,k}$ with $j\neq k$ are distinct minimal forts. Consequently,
    $$
        f_m(G) \geq \binom{s}{2} + \sum_{i=1}^{s}\binom{t_i}{2}.
    $$
    Since $n\leq s+2T$, we have
    $$
        3f_m(G) \geq 3\binom{s}{2}+3\sum_{i=1}^{s}\binom{t_i}{2} \geq 3\binom{s}{2}+2T-2s \geq s+2T \geq n.
    $$
    Therefore $f_m(G)\geq n/3$, completing the induction.
\end{proof}

\section{Counting Compatible Collections of Forts}

Hicks et al.~\cite{hicks2022computational} posed as their Question~1 the problem of characterizing the graphs that admit polynomially many compatible sets of minimal forts. Our work is shaped by a subsequent correction. Furst, Hutchens, Mitchell, and Zhang~\cite{furst2025compatible} observed that the minimum-rank bound motivating that question relies on the assumption that the minimal supports of null vectors of a matrix $A\in\operatorname{Sym}(G)$ coincide with the minimal forts of $G$, and they exhibited the Petersen graph as a counterexample: every sub-collection of $\operatorname{Fort}_{\min}(P)$ attaining the transversal number $5$ fails to be compatible. To repair the bound they replace minimal forts by arbitrary forts, defining the fort transversal number $T(G)=\max\{\tau(\mathcal F):\mathcal F\subseteq\operatorname{Fort}(G)\text{ compatible}\}$ and proving $N(G)\le T(G)\le Z(G)$. Since it is this corrected quantity that the maximum-nullity and minimum-rank application requires, we study the analogue of Question~1 for compatible collections of arbitrary forts.

\begin{question}\label{q1:refined}
    For which graphs is the number of compatible sets of forts bounded by a polynomial in the order of the graph?
\end{question}

Our results show that the answer is negative on three broad classes: trees, connected non-star graphs of bounded maximum degree, and connected graphs of minimum degree linear in the order. The mechanism is that the quantity being bounded already dominates the number of forts itself, which we show is exponential on each of these classes. Whether a family with unbounded maximum degree and sublinear minimum degree can behave differently we do not know. The contrast with the count of \emph{minimal} forts is worth keeping in view: that quantity can be polynomial, and its extremal behaviour was studied by Becker,
Cameron, Hanely, Ong, and Previte~\cite{becker2025number}.

We write $V_{\geq 2}$ for the vertices of degree at least $2$, $L$ for the number of leaves, $\delta(G)$ for the minimum degree, $\alpha(G)$ for the independence number, and $\operatorname{dom}(G)$ for the number of dominating sets of a graph $G$.

It is convenient to phrase the fort condition through complements. For $S\subseteq V$ put $F=V\setminus S$. Then $F$ is a fort if and only if $S\neq V$ and no $s\in S$ has exactly one neighbor outside $S$. We call such a proper subset $S$ \emph{admissible}. Equivalently, by a result of Becker, Cameron, Hanely, Ong, and Previte~\cite[Thm.~10]{becker2025number}, admissible sets are the failed zero forcing sets of $G$, and minimal forts are complements of maximal failed zero forcing sets.

\begin{lemma}\label{lem:singleton}
    For every graph $G$, $\operatorname{cs}(G) \geq f(G)$ and $\operatorname{cs}_m(G)\geq f_m(G)$.
\end{lemma}

\begin{proof}
    Every singleton collection of forts is compatible, and distinct forts give distinct singletons.
\end{proof}

By Lemma~\ref{lem:singleton}, polynomially many compatible sets of forts forces polynomially many forts. The rest of the section bounds $f(G)$ from below on three classes of graphs, in each case exponentially.

\begin{lemma}\label{lem:A}
    If $J\subseteq V_{\geq 2}$ is independent, then $V\setminus A$ is a fort for every $A\subseteq J$, and hence $f(G)\geq 2^{|J|}$. In particular, taking $J$ to be a maximum independent set of $G[V_{\geq 2}]$ gives $f(G)\geq 2^{\alpha(G[V_{\geq 2}])}$.
\end{lemma}

\begin{proof}
    For $S=A\subseteq J$ and $s\in A$, independence forces every neighbor of $s$ outside $A=S$, so $|N(s)\cap(V\setminus S)|=\deg(s)\ge 2$. $F$ is a fort if and only if no $s\in S$ has exactly one neighbor outside $S$, yielding the result.
\end{proof}

\begin{lemma}\label{lem:B}
    Let $J$ be an independent set of support vertices, each with a non-leaf neighbor, and let $k_v$ be the number of leaf-neighbors of $v \in J$. Then $f(G) \geq 2^{\sum_{v\in J}k_v}$.
\end{lemma}

\begin{proof}
    For $v\in J$, let $d_v$ be the number of non-leaf neighbors of $v$; by hypothesis $d_v\geq1$. Independently for each $v\in J$, make one of the following choices:
    \begin{itemize}
        \item put neither $v$ nor any of its leaf neighbors in $S$; or
        \item put $v$ in $S$, choose a set $A_v$ of its leaf neighbors to put in $S$, subject to $d_v+k_v-|A_v|\neq1$.
    \end{itemize}
    No non-leaf neighbor of a vertex of $J$ is placed in $S$: such a neighbor is not a leaf of any vertex of $J$, and independence of $J$ prevents it from being another selected support vertex adjacent to $v$. Thus, when $v\in S$, $|N(v)\cap(V\setminus S)|=d_v+k_v-|A_v|\neq 1$. Every leaf placed in $S$ has its unique neighbor $v$ also in $S$, and therefore has zero neighbors outside $S$. Hence $S$ is admissible. If $d_v=1$, only the choice $A_v$ consisting of all $k_v$ leaves is forbidden; if $d_v\geq2$, no subset is forbidden. Including the first option, there are at least $2^{k_v}$ choices at $v$. Multiplying over $v\in J$ gives $f(G)\geq\prod_{v\in J}2^{k_v}=2^{\sum_{v\in J}k_v}$.
\end{proof}

\begin{lemma}\label{lem:C}
    If every vertex of $C\subseteq V$ has at least $2$ neighbors outside $C$, then $V \setminus S$ is a fort for every $S\subseteq C$; hence $f(G) \geq 2^{|C|}$. In particular, if $\delta(G) \geq d$ then every $S$ with $|S| \leq d-1$ is admissible, so $f(G) \ge \sum_{j=0}^{d-1} \binom nj$.
\end{lemma}

\begin{proof}
    For $S \subseteq C$ and $s \in S$, the $\geq 2$ neighbors of $s$ outside $C$ lie outside $S$, so $|N(s) \cap (V \setminus S)| \geq 2$. For the minimum-degree statement, if $|S| \leq d-1$ then any $s \in S$ has $|N(s) \cap (V \setminus S)| \geq \deg(s)-(|S|-1) \geq d-(d-2) = 2$.
\end{proof}

\begin{lemma}\label{lem:D}
    If $V\setminus C$ is a dominating set of $G$, then $f(G) \geq \operatorname{dom}(G[C])$.
\end{lemma}

\begin{proof}
    Take $S \subseteq C$ with $C \setminus S$ a dominating set of $G[C]$, and put $F = V \setminus S$. Each $s \in S$ has a neighbor in $V \setminus C \subseteq F$ as $V \setminus C$ dominates $G$, and since $s \notin C \setminus S$, a neighbor in $C \setminus S \subseteq F$ as $C\setminus S$ dominates $G[C]$. These are distinct, so $|N(s) \cap F| \geq 2$. Since $F$ is a fort if and only if no $s\in S$ has exactly one neighbor outside $S$, we find that $S$ is admissible, and distinct dominating sets $C\setminus S$ give distinct forts.
\end{proof}

First note that Lemmas~\ref{lem:A}--\ref{lem:C} reproduce the exact counts of the extremal families.

\begin{proposition}\label{prop:exact}
    For the path, star, and complete graph we have:
    $$ 
        f(P_n)=F_n \quad (n\geq1), \quad \quad f(K_{1,n})=2^n-n\quad (n\geq1), \quad \text{and} \quad f(K_n)=2^n-n-1\quad (n\geq2).
    $$
    Moreover, $f(K_1)=1$.
\end{proposition}

\begin{proof}
    For $K_n$, $F$ is a fort if and only if $|F| \neq 1$. For $K_{1,n}$ the forts are the $\geq 2$-subsets of leaves together with $V$. For $P_n$, the endpoint condition $x_1=0\Rightarrow x_2=0$ cascades through $x_v=0\Rightarrow x_{v-1}=x_{v+1}$ to the zero vector, so every fort contains both endpoints; admissible sets are exactly the independent sets of the interior path, numbering $F_n$.

    Moreover, two distinct forts in $P_n$ share endpoint $1$, which no fort omits, so the compatibility clause fails and every compatible collection is a singleton. Therefore, $\operatorname{cs}(P_n) = f(P_n) = F_n$.
\end{proof}

\begin{theorem}\label{thm:tree}
    Every tree on $n \geq 3$ vertices satisfies $f(T) \geq 2^{n/6}$, hence $\operatorname{cs}(T) \geq 2^{n/6}$.
\end{theorem}

\begin{proof}
    If $T$ is a star, $f(T) = 2^{n-1}-(n-1) \geq 2^{n/6}$ for $n \geq 3$. Otherwise let $I = V_{\geq 2}$, which induces a sub-tree, and $L = n-|I|$. If $|I| \geq n/3$, the internal sub-tree is bipartite, so $\alpha(T[I]) \geq |I|/2 \geq n/6$ and Lemma~\ref{lem:A} gives $f(T) \geq 2^{n/6}$. If $|I| < n/3$, then $L > 2n/3$; every support vertex of a non-star tree has a non-leaf neighbor, so two-coloring the support forest and taking the heavier class $J$, Lemma~\ref{lem:B} gives $f(T) \geq 2^{L/2} > 2^{n/3}$.
\end{proof}

\begin{theorem}\label{thm:bdd}
    Every connected non-star graph of order $n$ and maximum degree $\Delta$ satisfies $f(G) \ge 2^{n/(2\Delta(\Delta+1))}$. Hence every bounded-degree family has exponentially many forts.
\end{theorem}

\begin{proof}
    As $G[V_{\geq 2}]$ has maximum degree $\leq \Delta$, $\alpha(G[V_{\geq 2}]) \geq (n-L)/(\Delta+1)$, so Lemma~\ref{lem:A} gives $f(G) \geq 2^{(n-L)/(\Delta+1)}$, at least $2^{n/(2(\Delta+1))}$ when $L \leq n/2$. If $L > n/2$, there are $\geq L/\Delta$ support vertices inducing a graph of maximum degree $\leq \Delta$, hence an independent set $J$ of $\geq L/(\Delta(\Delta+1))$ of them, each with a non-leaf neighbor, by Lemma~\ref{lem:B} gives $f(G) \geq 2^{|J|}$.
\end{proof}

The remaining classes are dense. We use Lemma~\ref{lem:D} in combination with the following classical bound of Wagner~\cite{wagner2013note}.

\begin{lemma}\label{lem:wagner}
    Every graph on $n$ vertices with no isolated vertex has at least $2^{n/2}$ dominating sets.
\end{lemma}

\begin{theorem}\label{thm:dense}
    Let $G$ be connected. If $\delta(G)=\Omega(n)$ then $f(G)=2^{\Omega(n)}$. More generally, if $G$ has a dominating set $D$ with $|D| \leq (1-\epsilon) n$ such that $G[V\setminus D]$ has an induced subgraph on $\Omega(n)$ vertices with no isolated vertex, then $f(G)\geq \operatorname{dom}(G[V\setminus D]) \geq 2^{\Omega(n)}$.
\end{theorem}

\begin{proof}
    For the first claim, suppose that $\delta(G)\geq cn$ for a fixed constant $c>0$ and all sufficiently large $n$. By Lemma~\ref{lem:C}, $f(G)\geq\sum_{j=0}^{\delta(G)-1}\binom{n}{j}$. Taking $j=\lfloor cn/2\rfloor$ gives $f(G)\geq\binom{n}{\lfloor cn/2\rfloor}=2^{\Omega(n)}$.
   
    For the second, take $C = V\setminus D$; since $D$ dominates, Lemma~\ref{lem:D} gives $f(G) \geq \operatorname{dom}(G[C])$. Isolated vertices of $G[C]$ lie in every dominating set of $G[C]$ and may be discarded without changing $\operatorname{dom}(G[C])$, so $\operatorname{dom}(G[C]) = \operatorname{dom}(G[C'])$ for the non-isolated part $C'$; by Lemma~\ref{lem:wagner}, $\operatorname{dom}(G[C']) \geq 2^{|C'|/2}=2^{\Omega(n)}$.
\end{proof}

\begin{corollary}\label{cor:resolution}
    For every connected graph in the union of the classes of Theorems~\ref{thm:tree}, \ref{thm:bdd}, and \ref{thm:dense}, $\operatorname{cs}(G)=2^{\Omega(n)}$. In particular no family of trees, of bounded-degree graphs, or of dense graphs as above has polynomially many compatible sets of forts.
\end{corollary}

\section{Counting Zero Forcing Minimal Forts on Trees}

Bounds on the number of zero forcing minimal forts of a tree, together with an algorithm for computing them, were obtained by Nguyen and Kenter~\cite{dat2026number}. We take a different route and give an algorithm that computes $f_m(T)$, the number of minimal forts of a tree $T$, exactly, in linear time and linear space.

Throughout this subsection $T=(V,E)$ is a tree on $n$ vertices rooted at a distinguished vertex $r\in V$, and $N(v)$ denotes the open neighborhood of $v$ in $T$. For $v\in V$ we write $\mathrm{ch}(v)$ for the set of children of $v$, $q_v=|\mathrm{ch}(v)|$, and $T_v$ for the sub-tree induced by $v$ together with all of its descendants. We fix once and for all an arbitrary ordering $c_1,\ldots,c_{q_v}$ of $\mathrm{ch}(v)$, and for $0\leq i\leq q_v$ we let $T_v^{\,i}$ denote the sub-tree induced by $v$ together with $V(T_{c_1}),\ldots,V(T_{c_i})$, so that $T_v^{\,0}$ consists of $v$ alone and $T_v^{\,q_v}=T_v$.

We begin with the effect of deleting a single vertex from a fort. Let $F$ be a fort of $T$ and let $v\in F$. A vertex $u\in V$ is a \emph{witness for the removal of $v$} if $u\in V\setminus(F\setminus\{v\})$ and $|N(u)\cap(F\setminus\{v\})|=1$. Thus, when $|F|\geq2$ the set $F\setminus\{v\}$ is nonempty and fails to be a fort precisely when a witness for the removal of $v$ exists; when $F=\{v\}$ the deletion is empty and is not a fort by definition, whether or not a witness exists.

\begin{lemma}\label{lem:witness}
    Let $F$ be a fort of $T$ and $v\in F$. Then a witness for the removal of $v$ is either $v$ itself or a vertex in $N(v)\setminus F$. More precisely, $u$ is a witness for the removal of $v$ if and only if one of the following holds:
	\begin{enumerate}[label=\upshape(\roman*)]
        \item\label{it:self} $u=v$ and $|N(v)\cap F|=1$; or
		\item\label{it:nbr}  $u\in N(v)\setminus F$ and $|N(u)\cap F|=2$.
	\end{enumerate}
\end{lemma}

\begin{proof}
	Let $F' = F \setminus \{v\}$. By definition $u$ is a witness for the removal of $v$ exactly when $u \in V \setminus F'$ and $|N(u) \cap F'| = 1$. Since $V \setminus F' = (V \setminus F) \cup \{v\}$, every candidate $u$ falls into one of two cases. If $u = v$, as $v \notin F'$, $v$ is a candidate, and since $v \notin N(v)$ we have $N(v) \cap F' = N(v) \cap F$. Hence $v$ is a witness if and only if $|N(v) \cap F| = 1$.

	If $u \in V \setminus F$, because $F$ is a fort, $|N(u) \cap F| \neq 1$. Deleting $v$ gives $|N(u) \cap F'| = |N(u) \cap F| - \mathbf{1}[v \in N(u)]$. For $u$ to be a witness we need $|N(u) \cap F'| = 1$; since the count can drop by at most one, this forces $v \in N(u)$ and $|N(u) \cap F| = 2$, i.e. $u \in N(v) \setminus F$ with $|N(u) \cap F| = 2$.

	Conversely both conditions clearly produce a witness. Finally, any $u \in V \setminus F$ with $v \notin N(u)$ has $|N(u) \cap F'| = |N(u) \cap F| \neq 1$ and is never a witness, so the two cases are exhaustive.
\end{proof}

Quantifying Lemma~\ref{lem:witness} over all vertices of $F$ yields a necessary condition for minimality, stated here for an arbitrary graph.

\begin{proposition}\label{prop:minchar}
    Let $F$ be a fort of a graph $G$ with $|F|\geq2$. Then no set of the form $F\setminus\{v\}$ with $v\in F$ is a fort of $G$ if and only if every $v\in F$ satisfies at least one of the following: either $|N(v)\cap F|=1$, or some $u\in N(v)\setminus F$ satisfies $|N(u)\cap F|=2$. In particular, every minimal fort with at least two vertices satisfies this condition at each of its vertices.
\end{proposition}

\begin{proof}
	By Lemma~\ref{lem:witness} the set $F\setminus\{v\}$ fails to be a fort exactly when $v$ admits a witness, and the two alternatives in the statement are the two kinds of witness. A minimal fort with at least two vertices has no fort among its proper subsets, hence none among its co-singleton subsets.
\end{proof}

The condition of Proposition~\ref{prop:minchar} is necessary but not sufficient, and the reason is that forts are closed under unions: if $F_1$ and $F_2$ are forts and $u\notin F_1\cup F_2$, then $|N(u)\cap F_1|\neq1$ and $|N(u)\cap F_2|\neq1$ force $|N(u)\cap(F_1\cup F_2)|\neq1$, since the latter count vanishes when both counts vanish and is at least two otherwise. A fort may therefore contain a fort properly while every one of its vertices still admits a witness.

\begin{remark}\label{ex:notminimal}
	Let $T$ be the double star obtained from an edge $c_1c_2$ by attaching leaves $x_1,x_2,x_3$ to $c_1$ and $y_1,y_2,y_3$ to $c_2$, and let $F=\{x_1,x_2,y_1,y_2\}$. Since $|N(c_1)\cap F|=|N(c_2)\cap F|=2$, every vertex of $F$ admits a witness, so no co-singleton subset of $F$ is a fort; nevertheless $F$ contains the fort $\{x_1,x_2\}$ properly. The phenomenon is not confined to unions of disjoint forts: in the tree with edges $x_1s$, $x_2s$, $st$ and $tz$, the set $F=\{x_1,x_2,s,z\}$ is a fort, admits a witness at each of its vertices, is not a disjoint union of two forts, and properly contains the fort $\{x_1,x_2\}$, which is obtained from it by deleting $s$ and $z$ simultaneously.
\end{remark}

Minimality must therefore be certified against the deletion of arbitrary subsets, and this is what the following reformulation records.

\begin{definition}\label{def:split}
	Let $F$ be a fort of a graph $G$. A \emph{split} of $F$ is an ordered partition $F=D\sqcup K$ into a \emph{deleted part} $D$ and a \emph{kept part} $K$. The split is \emph{valid} if $D\neq\emptyset$, $K\neq\emptyset$, and $K$ is a fort of $G$.
\end{definition}

\begin{lemma}\label{lem:split}
	A fort $F$ of a graph $G$ is minimal if and only if $F$ admits no valid split.
\end{lemma}

\begin{proof}
	If $F=D\sqcup K$ is a valid split, then $K$ is a fort with $\emptyset\neq K\subsetneq F$, so $F$ is not minimal. Conversely, if $F$ is not minimal, then some fort $K$ satisfies $\emptyset\neq K\subsetneq F$, and $D=F\setminus K$ yields a valid split.
\end{proof}

Lemma~\ref{lem:witness} is the case $|D|=1$ of Lemma~\ref{lem:split}: a fort admits a valid split whose deleted part is a singleton if and only if one of its vertices admits no witness. The algorithm developed below tracks splits with deleted parts of arbitrary size, and the point that makes this possible in linear time is that a bounded amount of information about the boundary of a sub-tree suffices.

Let $F\subseteq V$ be a candidate vertex set. We call the pair $(F,T_v^{\,i})$ a \emph{configuration}, and we call it \emph{locally fortified} if $|N(u)\cap F|\neq1$ for every $u\in V(T_v^{\,i})\setminus(F\cup\{v\})$. Since $N(u)\subseteq V(T_v^{\,i})$ for every $u\in V(T_v^{\,i})\setminus\{v\}$, being locally fortified is a property of the configuration alone; and $F$ is a fort of $T$ if and only if $F\neq\emptyset$, the configuration $(F,T_r)$ is locally fortified, and either $r\in F$ or $|N(r)\cap F|\neq1$.

\begin{definition}\label{def:partial}
	A \emph{partial split} of a configuration $(F,T_v^{\,i})$ is an ordered partition $F\cap V(T_v^{\,i})=D'\sqcup K'$ such that $|N(u)\cap K'|\neq1$ for every $u\in V(T_v^{\,i})\setminus(K'\cup\{v\})$.
\end{definition}

The condition at $v$ itself, namely $|N(v)\cap K|\neq1$ for the eventual global kept part $K$, is deliberately omitted from Definition~\ref{def:partial}; we call it the \emph{pending condition}. It is the only condition of a global split that involves vertices outside $T_v^{\,i}$, and it is required only when $v\notin K$. Every other condition is intrinsic to a sub-tree, which is what makes a bottom-up computation possible.

\begin{definition}\label{def:type}
    The \emph{boundary type} of a partial split $(D',K')$ of $(F,T_v^{\,i})$ is the quadruple $\tau=(\sigma,\gamma,\delta,\kappa)$ in which $\sigma\in\{\mathrm{out},\mathrm{del},\mathrm{keep}\}$ records whether $v\notin F$, $v\in D'$, or $v\in K'$; in which $\gamma=\min\{|\{c_1,\ldots,c_i\}\cap K'|,\,2\}$ when $\sigma\neq\mathrm{keep}$ and $\gamma=0$ otherwise; and in which $\delta=\mathbf{1}[D'\neq\emptyset]$ and $\kappa=\mathbf{1}[K'\neq\emptyset]$.
\end{definition}

The counter $\gamma$ is capped at two because the pending condition compares $|N(v)\cap K|$ with one only, so that counts of two or more need not be distinguished; the same convention is applied to the counter $\lambda$ below. When $\sigma=\mathrm{keep}$ the pending condition is vacuous and the counter carries no information, which is why it is normalized to zero.

\begin{definition}\label{def:subsume}
	A boundary type $(\sigma,\gamma,\delta,\kappa)$ \emph{subsumes} a boundary type $(\sigma',\gamma',\delta',\kappa')$ if $\sigma=\sigma'$, $\delta\geq\delta'$, $\kappa\geq\kappa'$, and either $\gamma=\gamma'$ or $\gamma=2$.
\end{definition}

\begin{lemma}\label{lem:compress}
	Let $s$ and $s'$ be partial splits of the same configuration $(F,T_v^{\,i})$ such that the boundary type of $s$ subsumes the boundary type of $s'$. If some assignment of the vertices of $F\setminus V(T_v^{\,i})$ to the deleted and the kept class extends $s'$ to a valid split of $F$, then the same assignment extends $s$ to a valid split of $F$.
\end{lemma}

\begin{proof}
	Let $(D,K)$ and $(D',K')$ be the two extended partitions. They agree outside $V(T_v^{\,i})$, and $v$ lies in the same class in both because the two types share $\sigma$. We check that $K$ is a fort. For $u\in V(T_v^{\,i})\setminus(K\cup\{v\})$ the condition $|N(u)\cap K|\neq1$ is intrinsic to $T_v^{\,i}$ and holds because $s$ is a partial split. For $u\notin V(T_v^{\,i})$ with $u\notin K$ we have $N(u)\cap V(T_v^{\,i})\subseteq\{v\}$, so $|N(u)\cap K|=|N(u)\cap K'|$, which differs from one because $K'$ is a fort. There remains the case $u=v$ with $v\notin K$. Write $\gamma$ and $\gamma'$ for the counters of $s$ and $s'$. If $\gamma=2$, then $v$ has at least two kept children, so $|N(v)\cap K|\geq2$. Otherwise $\gamma=\gamma'\in\{0,1\}$, both counters are exact, and $|N(v)\cap K|=|N(v)\cap K'|\neq1$. Finally $\delta\geq\delta'$ and $\kappa\geq\kappa'$, together with the fact that the two partitions agree outside $V(T_v^{\,i})$, give $D\neq\emptyset$ and $K\neq\emptyset$.
\end{proof}

By Lemma~\ref{lem:compress} a subsumed boundary type may be discarded: for deciding whether a configuration extends to a valid split, only the set of achievable boundary types up to subsumption matters. That set is described by five bounded quantities.

\begin{definition}\label{def:state}
	Let $(F,T_v^{\,i})$ be a locally fortified configuration. Its \emph{state} is the quintuple $(b,\lambda,e,\alpha,M)$, where $b=\mathbf{1}[v\in F]$; $\lambda=\min\{|\{c_1,\ldots,c_i\}\cap F|,\,2\}$; $e=\mathbf{1}[F\cap V(T_v^{\,i})=\emptyset]$; $\alpha=1$ if some partial split places $v$ in the kept class and has a nonempty deleted class, and $\alpha=0$ otherwise; and $M\subseteq\{0,1,2\}$ is the set of counter values attained by the partial splits that place $v$ outside the kept class and have both classes nonempty, reduced by discarding the values $0$ and $1$ whenever $2\in M$.
\end{definition}

The reduction of $M$ is licensed by Lemma~\ref{lem:compress} and leaves five possibilities, namely $\emptyset$, $\{0\}$, $\{1\}$, $\{0,1\}$ and $\{2\}$. Since $e=1$ forces $b=0$, $\lambda=0$, $\alpha=0$ and $M=\emptyset$, since $b=1$ forces $e=0$, and since $b=0$ forces $\alpha=0$, at most $1+3\cdot5+3\cdot2\cdot5=46$ states occur. Storing $M$ as three bits, a state consists of one counter and five bits.

\begin{lemma}\label{lem:state}
	Every locally fortified configuration $(F,T_v^{\,i})$ admits the \emph{all-kept} partial split $(\emptyset,\,F\cap V(T_v^{\,i}))$ and the \emph{all-deleted} partial split $(F\cap V(T_v^{\,i}),\,\emptyset)$, and its set of achievable boundary types, reduced by subsumption, is obtained by discarding the subsumed elements of
	$$
    	\mathcal{T}(b,\lambda,e,\alpha,M) =
    	\begin{cases}
    		\{(\mathrm{out},0,0,0)\}, & b=0,\ e=1,\\[1mm]
    		\{(\mathrm{out},\lambda,0,1),\,(\mathrm{out},0,1,0)\} \cup\{(\mathrm{out},\gamma,1,1):\gamma\in M\}, & b=0,\ e=0,\\[1mm]
    		\{(\mathrm{keep},0,\alpha,1),\,(\mathrm{del},0,1,0)\} \cup\{(\mathrm{del},\gamma,1,1):\gamma\in M\}, & b=1,
    	\end{cases}
	$$
	where $(b,\lambda,e,\alpha,M)$ is the state of the configuration.
\end{lemma}

\begin{proof}
    The all-deleted partition is a partial split because its kept class is empty, so that every count $|N(u)\cap K'|$ vanishes; its type is $(\mathrm{del},0,1,0)$ when $b=1$, and $(\mathrm{out},0,\mathbf{1}[F\cap V(T_v^{\,i})\neq\emptyset],0)$ when $b=0$. The all-kept partition is a partial split because its conditions are exactly the conditions defining local fortification; its counter equals $\lambda$, and its type is $(\mathrm{keep},0,0,1)$ when $b=1$ and $(\mathrm{out},\lambda,0,\mathbf{1}[F\cap V(T_v^{\,i})\neq\emptyset])$ when $b=0$. If $e=1$, the two partitions coincide with $(\emptyset,\emptyset)$, whose type is $(\mathrm{out},0,0,0)$, and no other partition of the empty set exists. If $e=0$, then every partial split other than these two has both classes nonempty and therefore carries the flags $\delta=\kappa=1$; among these, the ones placing $v$ in the kept class all have the single type $(\mathrm{keep},0,1,1)$, whose presence is recorded by $\alpha$, and the ones placing $v$ outside the kept class are recorded, through their counters, by $M$. A type carrying the flags $(1,1)$ is subsumed only by a type carrying the same flags, so the reduction of $M$ in Definition~\ref{def:state} discards exactly the subsumed types among the latter, and the displayed family is the set of achievable types.
\end{proof}

The next lemma is the transition of the dynamic program: it merges the state of a configuration on $T_v^{\,i}$ with the state of a configuration on the sub-tree of the next child.

\begin{lemma}\label{lem:merge}
	Let $(F,T_v^{\,i})$ be locally fortified with state $(b,\lambda,e,\alpha,M)$, and let $(F,T_{c_{i+1}})$ be locally fortified with state $(b',\lambda',e',\alpha',M')$. Then $(F,T_v^{\,i+1})$ is locally fortified if and only if $b'=1$ or $\lambda'+b\neq1$, and in that case its counter is $\min\{\lambda+b',2\}$, its emptiness bit is $e\wedge e'$, and its set of achievable boundary types, reduced by subsumption, is obtained by discarding the subsumed elements of
	$$
	   \bigl\{\,\mu(\tau,\tau'):\ \tau\in\mathcal{T}(b,\lambda,e,\alpha,M),\ \tau'\in\mathcal{T}(b',\lambda',e',\alpha',M'),\ \sigma'=\mathrm{keep} \text{ or } \gamma'+\mathbf{1}[\sigma=\mathrm{keep}]\neq1\,\bigr\},
	$$
	where $\tau=(\sigma,\gamma,\delta,\kappa)$, $\tau'=(\sigma',\gamma',\delta',\kappa')$, and $\mu(\tau,\tau')$ is the type whose first entry is $\sigma$, whose counter is $0$ if $\sigma=\mathrm{keep}$ and $\min\{\gamma+\mathbf{1}[\sigma'=\mathrm{keep}],2\}$ otherwise, and whose flags are $\delta\vee\delta'$ and $\kappa\vee\kappa'$. The components $\alpha$ and $M$ of the new state are read off from this set as in Definition~\ref{def:state}.
\end{lemma}

\begin{proof}
	Local fortification of $(F,T_v^{\,i+1})$ imposes the conditions of the two smaller configurations together with the condition at $c_{i+1}$, which is no longer omitted. If $c_{i+1}\notin F$, then $|N(c_{i+1})\cap F|$ equals the number of children of $c_{i+1}$ lying in $F$, plus $b$; this count is exact when $\lambda'\in\{0,1\}$ and is at least two when $\lambda'=2$, so the requirement that it differ from one is equivalent to $\lambda'+b\neq1$. This proves the first assertion, and the updates of the counter and of the emptiness bit are immediate.

	For the types, restricting a partition of $F\cap V(T_v^{\,i+1})$ to $V(T_v^{\,i})$ and to $V(T_{c_{i+1}})$ yields a pair of partitions, and conversely every such pair glues to a partition. We claim that the glued partition is a partial split precisely when both restrictions are partial splits and the displayed side condition holds. Indeed, for $u\in V(T_v^{\,i})\setminus\{v\}$ and for $u\in V(T_{c_{i+1}})\setminus\{c_{i+1}\}$ the conditions of Definition~\ref{def:partial} are intrinsic to the respective sub-tree, because $N(u)$ is contained in it. The only condition present in $T_v^{\,i+1}$ but omitted from both restrictions is the one at $c_{i+1}$ in the case that $c_{i+1}$ lies outside the kept class, and it reads $|N(c_{i+1})\cap K|\neq1$; its left-hand side is the child's counter plus $\mathbf{1}[v\in K]$, exactly below the cap and at least two at the cap, so the condition is equivalent to $\gamma'+\mathbf{1}[\sigma=\mathrm{keep}]\neq1$. The boundary type of the glued partial split is $\mu(\tau,\tau')$: the class of $v$ is unchanged, the counter increases by one exactly when $c_{i+1}$ is kept, and each class of the union is nonempty exactly when it is nonempty in one of the two parts.

	Finally, gluing respects subsumption. Fix $\tau$ and let $\tau'$ subsume $\tau''$. If the pair $(\tau,\tau'')$ satisfies the side condition, then so does $(\tau,\tau')$, because $\sigma'=\sigma''$ and either $\gamma'=\gamma''$ or $\gamma'=2$, and a counter equal to two satisfies the side condition outright; moreover $\mu(\tau,\tau')$ subsumes $\mu(\tau,\tau'')$, since the classes agree, the counters agree or the former is capped, and the flags are monotone. The symmetric statement holds in the first argument. Hence computing with the reduced sets of Lemma~\ref{lem:state} produces the reduced set of the union, and by Lemma~\ref{lem:compress} no information about extendability is lost.
\end{proof}

For $i=0$ the configuration $(F,T_v^{\,0})$ is locally fortified, and its state is $(1,0,0,0,\emptyset)$ if $v\in F$ and $(0,0,1,0,\emptyset)$ if $v\notin F$: the only partial splits of a single vertex are $(\{v\},\emptyset)$ and $(\emptyset,\{v\})$ in the first case and $(\emptyset,\emptyset)$ in the second, with the types prescribed by Lemma~\ref{lem:state}.

\begin{lemma}\label{lem:root}
	Let $(F,T_r)$ be locally fortified with state $(b,\lambda,e,\alpha,M)$. Then $F$ is a minimal fort of $T$ if and only if $b=1$ or $\lambda\neq1$; not both $b=0$ and $e=1$; $\alpha=0$; and $M\subseteq\{1\}$.
\end{lemma}

\begin{proof}
	The first condition says that $r\in F$ or $|N(r)\cap F|\neq1$, which together with local fortification says exactly that every vertex outside $F$ has a number of neighbors in $F$ different from one; the second says $F\neq\emptyset$. Jointly they say that $F$ is a fort, and by Lemma~\ref{lem:split} it remains to prove that, for a fort, the last two conditions say that no valid split exists. A valid split of $F$ is a partial split of $(F,T_r)$ with both classes nonempty that satisfies in addition the pending condition at $r$; as $r$ has no parent, that condition reads $|N(r)\cap K|\neq1$ and holds automatically when $r\in K$. A partial split placing $r$ in the kept class with both classes nonempty exists precisely when $\alpha=1$. A partial split placing $r$ outside the kept class with both classes nonempty and counter $\gamma$ satisfies the pending condition precisely when $\gamma\neq1$, and such a split exists precisely when $M$ contains $0$ or $2$. Because the reduction of $M$ discards the value $0$ only in the presence of the value $2$, the reduced set contains $0$ or $2$ if and only if the unreduced one does, and this is the negation of $M\subseteq\{1\}$.
\end{proof}

The dynamic program maintains, for every vertex $v$ and every state $s$, the number $A[v][s]$ of sets $F\cap V(T_v)$ whose configuration on $T_v$ is locally fortified and has state $s$. Lemmas~\ref{lem:state} and~\ref{lem:merge} show that the state of a configuration on $T_v^{\,i+1}$ is determined by the states of its two restrictions, so that configurations sharing a state may be aggregated and carried through the computation as a single count. The procedure is given as Algorithm~\ref{alg:main} in Appendix~\ref{app:algorithm}.

\begin{theorem}\label{thm:dp-alg}
    Let $T$ be a tree on $n$ vertices. Algorithm~\ref{alg:main} computes $f_m(T)$ using $O(n)$ additions and multiplications of integers with $O(n)$ bits, together with $O(1)$ table entries per vertex. In the unit-cost model it therefore runs in $O(n)$ time and $O(n)$ space.
\end{theorem}

\begin{proof}
	We show by induction on the height of $v$ that the table $A[v]$ is correct. For $T_v^{\,0}$ the two configurations and their states are those computed above, which is what the initialization records. For the inductive step, Lemma~\ref{lem:merge} determines which pairs consisting of a configuration on $T_v^{\,i}$ and a configuration on $T_{c_{i+1}}$ combine into a locally fortified configuration on $T_v^{\,i+1}$, and determines the resulting state; since $V(T_v^{\,i+1})$ is the disjoint union of $V(T_v^{\,i})$ and $V(T_{c_{i+1}})$, every configuration on $T_v^{\,i+1}$ arises by restriction from exactly one such pair, and the counts of the two parts multiply. Summing the products over all compatible pairs of states with a common merged state is therefore correct, and after $q_v$ steps the table $A[v]$ is correct. By Lemma~\ref{lem:root}, summing $A[r][s]$ over the accepting states $s$ counts every minimal fort exactly once and nothing else.

	For the complexity, each incorporation of a child ranges over pairs drawn from a state space of size at most $46$ and performs a bounded number of operations on each pair, since $\mathcal{T}$ and $\mu$ are functions between sets of bounded size and may be precomputed as lookup tables. The number of incorporations is $\sum_{v\in V}q_v=n-1$, and computing a post-order traversal costs $O(n)$, so the algorithm performs $O(n)$ arithmetic operations. Every table entry counts subsets of $V(T)$ and is therefore bounded by $2^{n}$, so each occupies $O(n)$ bits.
\end{proof}

\begin{remark}\label{rem:paths}
	Paths provide a closed form against which the output may be checked. Every nonempty fort of $P_n$ contains both endpoints, and a vertex set containing them is a fort if and only if its complement is an independent set of the interior path on $n-2$ vertices; the correspondence reverses inclusion, so the minimal forts of $P_n$ are the complements of the maximal independent sets of $P_{n-2}$. Writing $a_m$ for the number of maximal independent sets of $P_m$, one has $a_m=a_{m-2}+a_{m-3}$ with $a_0=a_1=1$ and $a_2=2$, whence $f_m(P_n)=a_{n-2}$ takes the values $1,1,2,2,3,4,5,7,9$ for $n=2,\ldots,10$. For the double star of Remark~\ref{ex:notminimal} the six minimal forts are the six pairs of leaves sharing a support vertex.
\end{remark}
 
\section*{Acknowledgments}

Aida Abiad is supported by NWO (Dutch Research Council) through grant VI.Vidi.213.085.

\section*{Declaration of AI Use}

During the preparation of this manuscript, the authors used Anthropic's Claude Fable 5 model to verify the results and to improve the clarity and quality of the writing.


\printbibliography

@article{cameron2025minimal,
  title={On the minimal forts of trees},
  author={Cameron, Thomas R and Li, Kelvin},
  journal={arXiv preprint arXiv:2512.12874},
  year={2025}
}

@article{jacob2025well,
  title={Well-failed graphs},
  author={Jacob, Bonnie},
  journal={arXiv preprint arXiv:2501.19357},
  year={2025}
}

@article{dat2026number,
  title={On the Number of Zero Forcing Minimal Forts on Trees},
  author={Dat, Nguyen Hoang and Kenter, Franklin HJ},
  journal={arXiv preprint arXiv:2605.07298},
  year={2026}
}

@article{hicks2022computational,
  title={Computational and theoretical challenges for computing the minimum rank of a graph},
  author={Hicks, Illya V and Brimkov, Boris and Deaett, Louis and Haas, Ruth and Mikesell, Derek and Roberson, David and Smith, Logan},
  journal={INFORMS Journal on Computing},
  volume={34},
  number={6},
  pages={2868--2872},
  year={2022},
  publisher={INFORMS}
}

@article{furst2025compatible,
  title={Compatible Forts and Maximum Nullity of a Graph},
  author={Furst, Veronika and Hutchens, John and Mitchell, Lon and Zhang, Yaqi},
  journal={Graphs and Combinatorics},
  volume={41},
  number={3},
  pages={56},
  year={2025},
  publisher={Springer}
}

@article{becker2025number,
  title={On the number of minimal forts of a graph},
  author={Becker, Paul and Cameron, Thomas R and Hanely, Derek and Ong, Boon and Previte, Joseph P},
  journal={Graphs and Combinatorics},
  volume={41},
  number={1},
  pages={25},
  year={2025},
  publisher={Springer}
}

@article{wagner2013note,
  title={A note on the number of dominating sets of a graph},
  author={Wagner, Stephan},
  journal={Util. Math},
  volume={92},
  pages={25--31},
  year={2013}
}

@article{brimkov2021improved,
  title={Improved computational approaches and heuristics for zero forcing},
  author={Brimkov, Boris and Mikesell, Derek and Hicks, Illya V},
  journal={INFORMS Journal on Computing},
  volume={33},
  number={4},
  pages={1384--1399},
  year={2021},
  publisher={INFORMS}
}

@article{aim2008zero,
  title={Zero forcing sets and the minimum rank of graphs},
  author={AIM Minimum Rank--Special Graphs Work Group and others},
  journal={Linear algebra and its applications},
  volume={428},
  number={7},
  pages={1628--1648},
  year={2008},
  publisher={Elsevier}
}

@article{burgarth2007full,
  title={Full control by locally induced relaxation},
  author={Burgarth, Daniel and Giovannetti, Vittorio},
  journal={Physical review letters},
  volume={99},
  number={10},
  pages={100501},
  year={2007},
  publisher={APS}
}

@article{yang2013fast,
  title={Fast--mixed searching and related problems on graphs},
  author={Yang, Boting},
  journal={Theoretical Computer Science},
  volume={507},
  pages={100--113},
  year={2013},
  publisher={Elsevier}
}

@article{fallat2016complexity,
  title={On the complexity of the positive semidefinite zero forcing number},
  author={Fallat, Shaun and Meagher, Karen and Yang, Boting},
  journal={Linear Algebra and its Applications},
  volume={491},
  pages={101--122},
  year={2016},
  publisher={Elsevier}
}

@article{brimkov2019computational,
  title={Computational approaches for zero forcing and related problems},
  author={Brimkov, Boris and Fast, Caleb C and Hicks, Illya V},
  journal={European Journal of Operational Research},
  volume={273},
  number={3},
  pages={889--903},
  year={2019},
  publisher={Elsevier}
}

@article{aazami2008hardness,
  title={Hardness results and approximation algorithms for some problems on graphs},
  author={Aazami, Ashkan},
  year={2008},
  publisher={University of Waterloo}
}


\newpage
\appendix

\section{Appendix: Pseudocode and Implementation Details}\label{app:algorithm}

In the pseudocode below, $\mathcal{T}(s)$ denotes the reduced type set of the state $s$ as given by Lemma~\ref{lem:state}, $\mu$ denotes the merge operation of Lemma~\ref{lem:merge}, and $\alpha(S)$ and $M(S)$ denote the two components read off from a reduced type set $S$ as in Definition~\ref{def:state}. Both $\mathcal{T}$ and $\mu$ range over sets of bounded size and are precomputed once as lookup tables indexed by the at most $46$ states.

\begin{algorithm}[H]
	\caption{Counting the minimal forts of a tree}
	\label{alg:main}
	\begin{algorithmic}[1]
		\Require a tree $T$ rooted at $r$, with the children of each vertex ordered arbitrarily
		\Ensure the number $f_m(T)$ of minimal forts of $T$
		\State compute a post-order traversal of $T$
		\ForEach{vertex $v$ in post-order}
			\State $B\gets\mathbf{0}$;\quad $B[(1,0,0,0,\emptyset)]\gets1$;\quad $B[(0,0,1,0,\emptyset)]\gets1$
			\ForEach{child $c$ of $v$, in the fixed order}
				\State $B'\gets\mathbf{0}$
				\ForEach{state $s=(b,\lambda,e,\alpha,M)$ with $B[s]>0$}
					\ForEach{state $s'=(b',\lambda',e',\alpha',M')$ with $A[c][s']>0$}
						\If{$b'=0$ and $\lambda'+b=1$}
							\State \textbf{continue}
						\EndIf
						\State $S\gets\bigl\{\mu(\tau,\tau'):\tau\in\mathcal{T}(s),\ \tau'\in\mathcal{T}(s'),\ \sigma'=\mathrm{keep}\text{ or }\gamma'+\mathbf{1}[\sigma=\mathrm{keep}]\neq1\bigr\}$
						\State discard from $S$ every subsumed type
						\State $e''\gets e\wedge e'$;\quad $s''\gets\bigl(b,\ \min\{\lambda+b',2\},\ e'',\ \alpha(S),\ M(S)\bigr)$
						\State $B'[s'']\ \mathrel{{+}{=}}\ B[s]\cdot A[c][s']$
					\EndFor
				\EndFor
				\State $B\gets B'$
			\EndFor
			\State $A[v]\gets B$
		\EndFor
		\State $f\gets0$
		\ForEach{state $s=(b,\lambda,e,\alpha,M)$ with $A[r][s]>0$}
			\If{$(b=1$ or $\lambda\neq1)$ and not $(b=0$ and $e=1)$ and $\alpha=0$ and $M\subseteq\{1\}$}
				\State $f\ \mathrel{{+}{=}}\ A[r][s]$
			\EndIf
		\EndFor
		\State \Return $f$
	\end{algorithmic}
\end{algorithm}

An open-source Python implementation of Algorithm~\ref{alg:main} is publicly available at the GitHub repository \url{https://github.com/sinaqane/minimal-forts}. It was verified against a brute-force enumeration that lists all forts of a tree and extracts minimal ones by comparing each fort with all forts of smaller cardinality, exhaustively over all trees with at most eleven vertices and every choice of root, as well as several thousand random trees with up to sixteen vertices.

\end{document}